\documentclass[graybox]{svmult}

\usepackage{mathptmx}       % selects Times Roman as basic font
\usepackage{helvet}         % selects Helvetica as sans-serif font
\usepackage{courier}        % selects Courier as typewriter font
\usepackage{type1cm}        % activate if the above 3 fonts are
 \usepackage{dsfont}
 \usepackage{amssymb}
\usepackage{makeidx}         % allows index generation
\usepackage{graphicx}        % standard LaTeX graphics tool
\usepackage{multicol}        % used for the two-column index
\usepackage[bottom]{footmisc}% places footnotes at page bottom

\renewcommand{\theequation}{\theequation. \arabic{equation}}
\newtheorem{thm}{Theorem}[section]

\newtheorem{prop}[thm]{Proposition}
\newtheorem{defn}[thm]{Definition}

\makeindex             % used for the subject index
\begin{document}

\title*{On a system of $q$-partial differential equations with applications to $q$-series}
\author{Zhi-Guo Liu}
\institute{Zhi-Guo Liu \at Department of Mathematics, East China Normal University, 
Shanghai 200241, P. R. China, \email{liuzg@hotmail.com }}

\maketitle

\abstract{Using the theory of functions of several variables and $q$-calculus, we prove an expansion theorem for the analytic function in several variables which satisfies a system of $q$-partial differential equations. Some curious applications of this expansion theorem to $q$-series are discussed.
In particular,  an extension of Andrews' transformation formula for the $q$-Lauricella function is given.}

\section{Introduction and preliminaries}
\label{sec:1}
The $q$-analogue of the partial differential equation
$f_x= f_y$ was studied  by us in \cite{Liu2010, LiuRam2013}.
This investigation led us to develop a systematic method of deriving $q$-formulas, and many interesting results in $q$-series are proved with this method \cite{Liu2010, Liu2011, LiuRam2013, LiuZeng, LiuSymmetry, Liu2016}.
In particular, in \cite[Proposition~10.2]{LiuRam2013} we establish a general $q$-transformation formula which includes Watson's $q$-analog of Whipple's theorem as a special case.

In \cite{LiuHahn}  we further investigated the $q$-analogue of the partial differential equation
$f_x=\alpha f_y$, where  $\alpha \not=0$.

In this paper we continue
our investigation to discuss the $q$-extension of the partial differential equation
$\beta f_x=\alpha f_y,$   where $\alpha, \beta$ are two nonzero complex numbers.

The motivation of this paper  is to find the $q$-extension of $f(\alpha x+\beta y)$ through a $q$-partial differential equation.

As usual, we use $\mathbb{C}$ to denote the set of all complex numbers and
$\mathbb{C}^k=\mathbb{C}\times \cdots \times \mathbb{C}$ the set of all $k$-dimensional
complex numbers.

For $a \in \mathbb{C}$ and  any positive integer $n$, we define the $q$-shifted factorial as
\[
(a; q)_n=\prod_{k=0}^{n-1}(1-aq^k),\nonumber
\]
and it is understood that $(a; q)_0=1.$  For $|q|<1$, we further define the $q$-shifted factorial $(a; q)_\infty$ by
\[
(a; q)_\infty=\lim_{n\to \infty} (a; q)_n=\prod_{k=0}^\infty (1-aq^k),
\]
since this infinite product converges when $|q|<1.$

If $m$ is a positive integer, we define the multiple
$q$-shifted factorial as follows
\[
(a_1, a_2,...,a_m;q)_n=(a_1;q)_n(a_2;q)_n ... (a_m;q)_n,
\]
where $n$ is an integer or $\infty$.

Unless otherwise stated,  we suppose throughout that $|q|<1,$ which ensures that all the sums and
products appear in the paper converge.

The basic hypergeometric series  or the $q$-hypergeometric series ${_{r}\phi_s}(\cdot)$ are defined as
\[
_{r}\phi_s \left({{a_1, \ldots, a_{r}} \atop {b_1, \ldots,
b_s}} ;  q, z  \right) =\sum_{n=0}^\infty
\frac{(a_1, \ldots, a_r; q)_n}{(q, b_1, \ldots, b_s; q)_n}
\left((-1)^n q^{n(n-1)/2}\right)^{1+s-r}z^n.
\]

We now introduce some basic concepts in the $q$-calculus. The $q$-derivative was introduced by Leopold Schendel \cite{Schendel} in 1877 and
Frank Hilton Jackson \cite{Jackson1908} in 1908, which is a $q$-analog of the ordinary derivative.
\begin{defn}\label{qddfn}
If $q$ is a complex number, then, for any function $f(x)$ of one variable, the  $q$-derivative of $f(x)$
with respect to $x,$ is defined as
\[
{D}_{q}\{f(x)\}=\frac{f(x)-f(qx)}{x},
\]
and we further define  ${D}_{q}^0  \{f\}=f,$ and for $n\ge 1$, ${D}_{q}^n \{f\}={D}_{q}\{{D}_{q}^{n-1}\{f\}\}.$
\end{defn}

Now we give the definitions of  the $q$-partial derivative and the $q$-partial differential equations.
\begin{defn}\label{qpdfn}
A $q$-partial derivative of a function of several variables is its $q$-derivative with respect to one of those variables, regarding other variables as constants. The $q$-partial derivative of a function $f$ with respect to the variable $x$ is denoted by $\partial_{q, x}\{f\}$.
\end{defn}

\begin{defn}\label{qpde}
A $q$-partial differential equation is an equation that contains unknown multivariable functions and their $q$-partial derivatives.
\end{defn}

The Gaussian binomial coefficients also called the $q$-binomial coefficients are
the $q$-analogs of the binomial coefficients, which
are given by
\[
{n\brack k}_q=\frac{(q; q)_n}{(q; q)_k(q; q)_{n-k}}.
\]

We now introduce the homogeneous  polynomials  $\Phi^{(\alpha, \beta)}_n(x, y|q)$ in the following definition.
\begin{defn} \label{hvjpolydefn} The homogeneous polynomials  $\Phi^{(\alpha, \beta)}_n(x, y|q)$  are defined by
\[
\Phi^{(\alpha, \beta)}_n(x, y|q) =\sum_{k=0}^n {n\brack k}_q (\alpha; q)_k (\beta; q)_{n-k} x^k y^{n-k}.
\]
\end{defn}

When $\alpha=\beta=0,$ the homogeneous polynomials $\Phi^{(\alpha, \beta)}_n(x, y|q)$ reduce to the homogeneous Rogers--Szeg\H{o} polynomials
(see, for example \cite{LiuRam2013})
\[
h_n(x, y|q)=\sum_{k=0}^n {n\brack k}_q  x^k y^{n-k}.
\]
If we  set $\beta=0$ in $\Phi^{(\alpha, \beta)}_n(x, y|q)$, we can obtain the homogeneous Hahn polynomials (see, for example \cite{LiuHahn})
which are given by
\[
\Phi^{(\alpha)}_n(x, y|q)=\sum_{k=0}^n {n\brack k}_q (\alpha; q)_k x^k y^{n-k}.
\]
If we set $\alpha=\beta$ in $\Phi^{(\alpha, \beta)}_n(x, y|q)$,
we can obtain the  homogeneous continuous $q$--Ultraspherical polynomials
\[
C_n(x, y; \beta|q)=\sum_{k=0}^n {n\brack k}_q (\beta; q)_k (\beta; q)_{n-k} x^k y^{n-k}.
\]

The polynomials  $\Phi^{(\alpha, \beta)}_n(x, 1|q)$ have been studied by  Verma and Jain \cite{VermaJain} among others. It is obvious that $\Phi^{(\alpha, \beta)}_n(x, y|q)=y^n \Phi^{(\alpha, \beta)}_n(x/y, 1|q)$, but
an important difference between $\Phi^{(\alpha, \beta)}_n(x, 1|q)$ and $\Phi^{(\alpha, \beta)}_n(x, y|q)$ is that the latter satisfies the following $q$-partial differential equation,
which does not appear in the literature before.
\begin{prop}\label{liupdepp} The homogeneous  polynomials $\Phi^{(\alpha, \beta)}_n(x, y|q)$ satisfy the
$q$-partial differential equation
\begin{eqnarray*}
&&\partial_{q, x}\left\{\Phi^{(\alpha, \beta)}_n(x, y|q)-\beta \Phi^{(\alpha, \beta)}_n(x, qy|q)\right\}\\
&&=\partial_{q, y}\left\{\Phi^{(\alpha, \beta)}_n(x, y|q)-\alpha \Phi^{(\alpha, \beta)}_n(qx, y|q)\right\}.
\end{eqnarray*}
\end{prop}
\begin{proof} \smartqed
Using the $q$-identity, $\partial_{q, x} \{ x^k\}=(1-q^k) x^{k-1},$  we easily deduce  that
\begin{eqnarray*}
&&\partial_{q, x}\left\{\Phi^{(\alpha, \beta)}_n(x, y|q)-\beta \Phi^{(\alpha, \beta)}_n(x, qy|q)\right\}\\
&&=\sum_{k=1}^n {n\brack k}_q (\alpha; q)_k (1-q^k) (\beta; q)_{n+1-k} x^{k-1} y^{n-k}.
\end{eqnarray*}
In the same way, using the identity, $\partial_{q, y} \{ y^{n-k}\}=(1-q^{n-k}) y^{n-k-1},$ we conclude that
\begin{eqnarray*}
&&\partial_{q, y}\left\{\Phi^{(\alpha, \beta)}_n(x, y|q)-\alpha \Phi^{(\alpha, \beta)}_n(qx, y|q)\right\}\\
&&=\sum_{k=0}^{n-1} {n\brack k}_q (\alpha; q)_{k+1} (\beta; q)_{n-k} (1-q^{n-k}) x^{k} y^{n-k-1}.
\end{eqnarray*}
If we make the variable change $k+1 \to k$ in the right-hand side of the above equation, we can find that
\begin{eqnarray*}
&&\partial_{q, y}\left\{\Phi^{(\alpha, \beta)}_n(x, y|q)-\alpha \Phi^{(\alpha, \beta)}_n(qx, y|q)\right\}\\
&&=\sum_{k=1}^{n} {n\brack {k-1}}_q (\alpha; q)_{k} (\beta; q)_{n+1-k}(1-q^{n-k+1}) x^{k-1} y^{n-k}.
\end{eqnarray*}
From the definition of the $q$-binomial coefficients, it is easy to verify that
\[
{n\brack k}_q (1-q^k)={n\brack {k-1}}_q  (1-q^{n-k+1}).
\]
Combining the above three equations, we complete the proof of Proposition~\ref{liupdepp}.
\qed
\end{proof}
\begin{defn} \label{qshiftdefn} If $f(x_1, \ldots, x_k)$ is a $k$-variable function,
then, for~$j=1, 2, \ldots,  k,$  the $q$-shift operator $\eta_{x_j}$ on the variable $x_j$ is defined as
\[
\eta_{x_j}\{f(x_1, \ldots, x_k)\}=f(x_1, \ldots, x_{j-1},  qx_j, x_{j+1}, \ldots, x_k)
\]

\end{defn}
The principal result of this paper is the following expansion theorem for the analytic functions in several variables.

\begin{thm}\label{pliuthm}
If $f(x_1,y_1, \ldots, x_k, y_k)$  is a $2k$-variable
 analytic function at $(0, \cdots, 0)\in \mathbb{C}^{2k}$, then,
 $f$ can be expanded in terms of
 \[
 \Phi_{n_1}^{(a_1, b_1)}(x_1, y_1|q)\cdots \Phi_{n_k}^{(a_k, b_k)}(x_k, y_k|q)
 \]
 if and only if  $f$ satisfies the following system of  $q$-partial differential equations:
 \[
 \partial_{q, x_j}(1-b_j \eta_{y_j})\{f\}=\partial_{q, y_j}(1-a_j \eta_{x_j})\{f\},
 \quad j\in\{1, 2, \ldots,  k\}.
 \]
\end{thm}
Theorem~\ref{pliuthm} is a very powerful tool for proving $q$-formulas, which allows us to derive some deep $q$-formulas.
The rest of the  paper is organized as follows. Section~\ref{sec:2} is devoted to the proof of  Theorem~\ref{pliuthm}.
In Section~\ref{sec:3}, we will illustrate our approach by using Theorem~\ref{pliuthm} to prove two $q$-formulas.
Andrews \cite{Andrews1972} proved the following transformation formula for the $q$-Lauricella function.
\begin{prop}\label{Andrewspp} For $\max\{ |a|, |c|,  |y_1|, \ldots, |y_k|\}<1,$ we have
\begin{eqnarray*}
\sum_{n_1, \ldots, n_k=0}^\infty \frac{(a; q)_{n_1+\cdots+n_k}(\beta_1; q)_{n_1}\cdots (\beta_k; q)_{n_k}
y_1^{n_1} \cdots y_k^{n_k}}
{(c; q)_{n_1+\cdots+n_k}(q; q)_{n_1}\cdots (q; q)_{n_k}}\\
=\frac{(a, \beta_1y_1, \ldots, \beta_k y_k; q)_\infty}{(c, y_1, \ldots, y_k; q)_\infty}
{_{k+1}\phi_{k}}\left({{c/a, y_1, \ldots, y_k}
\atop{\beta_1 y_1, \ldots, \beta_k y_k}}; q, a\right).
\end{eqnarray*}
\end{prop}
In Section~\ref{sec:4}, we extend the Andrews formula to the following $q$-formula by using Theorem~\ref{pliuthm}.
\begin{thm}\label{AndrewsLauricella}If $\max\{ |a|, |c|, |x_1|, |y_1|, \ldots, |x_k|, |y_k|\}<1,$ then, we have
\begin{eqnarray*}
&&\sum_{n_1, \ldots, n_k=0}^\infty \frac{(a; q)_{n_1+\cdots+n_k}\Phi^{(\alpha_1, \beta_1)}_{n_1}(x_1, y_1|q)
\cdots \Phi^{(\alpha_k, \beta_k)}_{n_k}(x_k, y_k|q)}
{(c; q)_{n_1+\cdots+n_k}(q; q)_{n_1}\cdots (q; q)_{n_k}}\\
&&\quad=\frac{(a, \alpha_1 x_1, \beta_1 y_1, \ldots \alpha_k x_k, \beta_k y_k; q)_\infty}{(c, x_1, y_1, \ldots, x_k, y_k; q)_\infty}\\
&&\qquad \times{_{2k+1}\phi_{2k}}\left({{c/a, x_1, y_1, \ldots, x_k, y_k}
\atop{\alpha_1 x_1, \beta_1 y_1, \ldots, \alpha_k x_k, \beta_k y_k}}; q, a\right).
\end{eqnarray*}
\end{thm}
When $x_1=x_2=\ldots=x_k=0,$ Theorem~\ref{AndrewsLauricella} immediately reduces to Andrews' formula in
Proposition~\ref{Andrewspp}, and when $\beta_1=\beta_2=\ldots=\beta_k=0$,  Theorem~\ref{AndrewsLauricella}
becomes \cite[Theorem~6.1]{LiuHahn}.  Theorem~\ref{AndrewsLauricella}  may be regarded as a multilinear generating function for $\Phi_n^{(a, b)}(x, y|q)$.
In Section~\ref{sec:5}, we prove another  multilinear generating function for $\Phi_n^{(a, b)}(x, y|q)$. In Section~\ref{sec:6}, we will use Theorem~\ref{pliuthm} to derive a $q$-integral formula.
%%%%%%%%%%%%%%%%%%%%%%%%%%%%%%%%%%%%%%%%%%%%%%%%%%%%%%%%%%%%%%%%%%%%%%%%%%%%%%%%%%%%%%%%%%%%%%%%%%%%%%%%%
\section{The proof of Theorem~\ref{pliuthm}}
\label{sec:2}
%%%%%%%%%%%%%%%%%%%%%%%%%%%%%%%%%%%%%%%%%%%%%%%%%%%%%%%%%%%%%%%%%%%%%%%%%%%%%%%%%%%%%%%%%%%%%%%%%%%%%%%%%%%%%
To prove Theorem~\ref{pliuthm},  we need the following fundamental property of
several complex variables (see, for example \cite[p. 5, Proposition~1]{Malgrange}).
\begin{prop}\label{mcvarapp}
If $f(x_1, x_2, \ldots,  x_k)$ is analytic at the origin $(0, 0, \ldots,  0)\in \mathbb{C}^k$, then,
$f$ can be expanded in an absolutely convergent power series,
 \[
 f(x_1, x_2, \ldots,  x_k)=\sum_{n_1, n_2, \ldots, n_k=0}^\infty \lambda_{n_1, n_2, \ldots, n_k}
 x_1^{n_1} x_2^{n_2}\cdots x_k^{n_k}.
 \]
 \end{prop}
 Now we begin to prove Theorem~\ref{pliuthm} with the help of Proposition~\ref{mcvarapp}.
\begin{proof}
\smartqed
This theorem can be proved  by induction. We first prove the theorem for the case
$k=1.$
Since $f$ is analytic at $(0, 0),$  we know that $f$ can be expanded in
an absolutely convergent power series in a neighborhood of $(0, 0)$. Thus there exists a sequence $\lambda_{k, l}$
independent of $x_1$ and $y_1$ such that
\begin{equation}
 f(x_1, y_1)=\sum_{k, l=0}^\infty \lambda_{k, l} (a_1; q)_k (b_1; q)_l {{k+l}\brack k}_q x_1^k y_1^l.
 \label{liu:eqn1}
\end{equation}
Substituting this into the $q$-partial differential equation,
$\partial_{q, x_1}\{f(x_1, y_1)-b_1 f(x_1, qy_1)\}=\partial_{q, y_1}\{f(x_1, y_1)-a_1 f(qx_1, y_1)\},$
using the identities,  $\partial_{q, x_1}\{x_1^k\}=(1-q^k)x_1^{k-1}$ and $\partial_{q, y_1}\{y_1^l\}=(1-q^l)y_1^{l-1}$,  we find that
\begin{eqnarray*}
\sum_{k, l=0}^\infty \lambda_{k, l} (a_1; q)_k (b_1; q)_{l+1} (1-q^k){{k+l}\brack k}_q x_1^{k-1} y_1^l\\
=\sum_{k, l=0}^\infty \lambda_{k, l} (a_1; q)_{k+1} (b_1; q)_l (1-q^l) {{k+l}\brack k}_q x_1^k y_1^{l-1}.
\end{eqnarray*}
Equating the coefficients of $x^{k-1}y^{l}$ on both sides of the above equation, we deduce that
\[
\lambda_{k, l}(a_1; q)_k (b_1; q)_{l+1}\frac{(q; q)_{k+l}}{(q; q)_{k-1}(q; q)_{l}}
=\lambda_{k-1, l+1} (a_1; q)_{k} (b_1; q)_{l+1} \frac{(q; q)_{k+l}}{(q; q)_{k-1}(q; q)_{l}}.
\]
It follows that $\lambda_{k, l}=\lambda_{k-1, l+1}.$ Iterating this relation $(k-1)$ times, we deduce that
$\lambda_{k, l}=\lambda_{0, l+k}.$ Substituting this into (\ref{liu:eqn1}), we arrive at
\[
 f(x_1, y_1)=\sum_{k, l=0}^\infty \lambda_{0, k+l} (a_1; q)_k (b_1; q)_l {{k+l}\brack k}_q x_1^k y_1^l.
\]
Making the variable change $n=k+l$ and interchanging the order of summation, we find that
\[
f(x_1, y_1)=\sum_{n=0}^\infty \lambda_{0, n} \sum_{k=0}^n {n\brack k}_q (a; q)_k (b; q)_{n-k} x^k y^{n-k}.
\]
Conversely, if $f(x_1, y_1)$ can be expanded in terms of $\Phi_n^{(a, b)}(x_1, y_1|q), $
then using Proposition~\ref{liupdepp},  we find that
\[
\partial_{q, x_1}\{f(x_1, y_1)-b_1 f(x_1, qy_1)\}=\partial_{q, y_1}\{f(x_1, y_1)-a_1 f(qx_1, y_1)\}.
\]
We conclude that Theorem~\ref{pliuthm} holds when $k=1$.

Now, we assume that the theorem is true for the case $k-1$ and consider the case $k$.
If we regard $f(x_1, y_1, \ldots, x_k, y_k)$ as a function of $x_1$ and $y_1, $ then $f$ is analytic at $(0, 0)$ and satisfies
\[
\partial_{q, x_1}\{f(x_1, y_1)-b_1 f(x_1, qy_1)\}=\partial_{q, y_1}\{f(x_1, y_1)-a_1 f(qx_1, y_1)\}.
\]
Thus,  there exists a sequence
$\{c_{n_1}(x_2, y_2, \ldots, x_k, y_k)\}$ independent of $x_1$ and $y_1$ such that
\begin{equation}
f(x_1, y_1, \ldots, x_k, y_k)=\sum_{n_1=0}^\infty c_{n_1}(x_2, y_2, \ldots, x_k, y_k)\Phi_{n_1}^{(a_1, b_1)}( x_1, y_1|q).
\label{liu:eqn2}
\end{equation}
Setting $x_1=0$ in the above equation and using $\Phi_{n_1}^{(a_1, b_1)}( 0, y_1|q)=(b_1; q)_{n_1} y_1^{n_1},$ we obtain
\[
f(0, y_1, x_2, y_2,  \ldots, x_k, y_k)=\sum_{n_1=0}^\infty (b_1; q)_{n_1}c_{n_1}(x_2, y_2, \ldots, x_k, y_k)y_1^{n_1}.
\]
Using the Maclaurin expansion theorem, we immediately deduce that
\[
c_{n_1}(x_2, y_2, \ldots, x_k, y_k)=\frac{\partial^{n_1} f(0, y_1, x_2, y_2,  \ldots, x_k, y_k)}{(b_1; q)_{n_1} n_1! \partial {y_1}^{n_1}}\Big|_{y_1=0}.
\]
Since $f(x_1, y_1, \ldots, x_k, y_k)$ is analytic near $(x_1, y_1, \ldots, x_k, y_k)=(0, \ldots, 0)\in \mathbb{C}^{2k},$  from
the above equation, we know that $c_{n_1}(x_2, y_2, \ldots, x_k, y_k)$ is analytic near $(x_2, y_2, \ldots, x_k, y_k)=(0, \ldots, 0)\in \mathbb{C}^{2k-2}.$
Substituting (\ref{liu:eqn2}) into the $q$-partial differential equations in Theorem~\ref{pliuthm}, we find that for $j=2, \ldots,  k,$
\begin{eqnarray*}
&&\sum_{n_1=0}^\infty  \partial_{q, x_j}(1-b_j \eta_{y_j})\{c_{n_1}(x_2, y_2, \ldots, x_k, y_k)\}\Phi_{n_1}^{(a_1, b_1)}(x_1, y_1|q)\\
&&=\sum_{n_1=0}^\infty\partial_{q, y_j}(1-a_j \eta_{x_j})\{c_{n_1}(x_2, y_2, \ldots, x_k, y_k)\}\Phi_{n_1}^{(a_1, b_1)}(x_1, y_1|q).
\end{eqnarray*}
By equating the coefficients of $\Phi_{n_1}^{(a_1, b_1)}(x_1, y_1|q)$ in the above equation, we find that
for $j=2, \ldots,  k,$
\begin{eqnarray*}
&\partial_{q, x_j}(1-b_j \eta_{y_j})\{c_{n_1}(x_2, y_2, \ldots, x_k, y_k)\}\\
&=\partial_{q, y_j}(1-a_j \eta_{x_j})\{c_{n_1}(x_2, y_2, \ldots, x_k, y_k)\}.
\end{eqnarray*}
Thus by the inductive hypothesis, there exists a sequence $\lambda_{n_1, n_2, \ldots, n_k}$ independent of
$x_2, y_2, \ldots, x_k, y_k$ (of course independent of $x_1$ and $y_1$) such that
\begin{eqnarray*}
&&c_{n_1}(x_2, y_2, \ldots, x_k, y_k)=\sum_{n_2, \ldots, n_k=0}^\infty \lambda_{n_1, n_2, \ldots, n_k}\\
&&\times \Phi_{n_2}^{(a_2, b_2)}( x_2, y_2|q)\ldots \Phi_{n_k}^{(a_k, b_k)}(x_k, y_k|q).
\end{eqnarray*}
Substituting this equation  into (\ref{liu:eqn2}), we complete the proof of the theorem.
\qed
\end{proof}
To determine if a given function is an analytic functions in several complex variables,
one can use the following theorem (see, for example, \cite[p. 28]{Taylor}).
\begin{thm}\label{hartogthm} {\rm (Hartogs' theorem).}
If a complex valued function $f(z_1, z_2, \ldots, z_n)$ is holomorphic (analytic) in each variable separately in a domain $U\in\mathbb{C}^n,$
then,  it is holomorphic (analytic) in $U.$
\end{thm}
%%%%%%%%%%%%%%%%%%%%%%%%%%%%%%%%%%%%%%%%%%%%%%%%%%%%%%%%%%%%%%%%%%%%%%%%%%%%%%%%%%%%%%%%%%%%%%%%%%%%%%%%%%%%%%%%
%%%%%%%%%%%%%%%%%%%%%%%%%%%%%%%%%%%%%%%%%%%%%%%%%%%%%%%%%%%%%%%%%%%%%%%%%%%%%%%%%%%%%%%%%%%%%%%%%
\section{ Some generating functions for $\Phi_n^{(a, b)}(x, y|q)$}
\label{sec:3}
%%%%%%%%%%%%%%%%%%%%%%%%%%%%%%%%%%%%%%%%%%%%%%%%%%%%%%%%%%%%%%%%%%%%%%%%%%%%%%%%%%%%%%%%%%%%%%%%%
The following proposition is equivalent to the formula \cite[Equation {(2.1)}]{VermaJain}. Verma and
Jain proved this formula by using many known results in $q$-series and the  technique  of interchanging the order of summation.
In this section, we will use Theorem~\ref{pliuthm} to prove this proposition and the proof is quite different from
that of Verma and Jain.
\begin{thm}\label{gthma} If $m \ge 0$ is an integer and $\max\{|xt|, |yt|\}<1,$ then, we have
\begin{eqnarray*}
\sum_{n=0}^\infty \Phi_{n+m}^{(a, b)}(x, y|q)\frac{t^{n+m}}{(q; q)_n}
=\frac{(axt, byt; q)_\infty}{(xt, yt; q)_\infty}
{_3\phi_2}\left({{q^{-m}, xt, yt}\atop{axt, byt}}; q, q \right).
\end{eqnarray*}
\end{thm}
\begin{proof} \smartqed
Denote the right-hand side of the above equation by $f(x, y).$ It is easily seen that $f(x, y)$ is
an analytic function of $x$ and $y$, for $\max\{|xt|, |yt|\}<1.$  Thus, $f(x, y)$ is analytic at $(0, 0)\in \mathbb{C}^2.$
A direct computation shows that
\begin{eqnarray*}
&&\partial_{q, x}\{f(x, y)-bf(x, qy)\}=\partial_{q, y}\{f(x, y)-af(qx, y)\}\\
&&=(1-a)(1-b)t \frac{(axtq, bytq; q)_\infty}{(xt, yt; q)_\infty}{_3\phi_2}\left({{q^{-m}, xt, yt}\atop{axtq, bytq}}; q, q^2 \right).
\end{eqnarray*}
 Thus, by Theorem~~\ref{pliuthm},  there exists a sequence $\{\lambda_n\}$ independent of $x$ and $y$ such that
 \[
\frac{(axt, byt; q)_\infty}{(xt, yt; q)_\infty}
{_3\phi_2}\left({{q^{-m}, xt, yt}\atop{axt, byt}}; q, q \right)=\sum_{n=0}^\infty \lambda_n \Phi_n^{(a, b)}(x, y|q).
 \]
 Putting $x=0$ in this equation, using $\Phi_n^{(a, b)}( 0, y|q)=(b; q)_n y^n$, we find that
  \[
\frac{( byt; q)_\infty}{(yt; q)_\infty}
{_2\phi_1}\left({{q^{-m},  yt}\atop{byt}}; q, q \right)=\sum_{n=0}^\infty \lambda_n (b; q)_n y^n.
 \]
 Equating the coefficients of $y^n$ and using the $q$-binomial theorem, we find that
 \begin{eqnarray*}
 \lambda_n=\frac{t^n}{(q; q)_n} \sum_{k=0}^m \frac{(q^{-m}; q)_k}{(q; q)_k} q^{k(1+n)}
 =\frac{t^n}{(q; q)_{n-m}}.
 \end{eqnarray*}
 It follows that
 \begin{eqnarray*}
 \sum_{n=0}^\infty \lambda_n \Phi_n^{(a, b)}(x, y|q)
 &&=\sum_{n=0}^\infty \frac{t^n}{(q; q)_{n-m}} \Phi_n^{(a, b)}(x, y|q)\\
 &&=\sum_{n=m}^\infty \frac{t^n}{(q; q)_{n-m}} \Phi_n^{(a, b)}(x, y|q).
 \end{eqnarray*}
 Making the variable change $n-m$ to $n,$ we complete the proof of the theorem.
 \qed
\end{proof}
When $m=0,$ Theorem~\ref{gthma} reduces to the following proposition, which can
be obtained easily by multiplying two copies of the $q$-binomial theorem together.
\begin{prop} \label{gthmb}If $\max\{|xt|, |yt|\}<1,$ then,  we have the formula
\[
\sum_{n=0}^\infty \Phi_n^{(a, b)} (x, y|q) \frac{t^n}{(q; q)_n}
=\frac{(axt, byt; q)_\infty}{(xt, yt; q)_\infty}.
\]
\end{prop}
Next we will use Theorem~\ref{pliuthm} to prove the following theorem.
\begin{thm}\label{gthmc} If $\max\{|c|, |ab|, |tx|, |ty|\}<1$ and $cd=ab,$ then, we have
\begin{eqnarray*}
\sum_{n=0}^\infty \frac{(c; q)_n \Phi_n^{(a, b)}(x, y|q)t^n}{(q, ab; q)_n}
=\frac{(c, axt, byt; q)_\infty}{(ab, tx, ty; q)_\infty}
{_3\phi_2 \left({{d, xt, yt}\atop{axt, byt}}; q, c \right)}.
\end{eqnarray*}
\end{thm}
\begin{proof}
\smartqed
If we use $f(x, y)$ to denote the right-hand side of the above equation, then, using the ratio test, we find that $f(x, y)$ is
an analytic function of $x$ and $y$, for $\max\{|c|, |ab|, |xt|, |yt|\}<1.$  Thus, $f(x, y)$ is analytic at $(0, 0)\in \mathbb{C}^2.$
A direct computation shows that
\begin{eqnarray*}
&&\partial_{q, x}\{f(x, y)-bf(x, qy)\}=\partial_{q, y}\{f(x, y)-af(qx, y)\}\\
&&=(1-a)(1-b)t \frac{(c, axtq, bytq; q)_\infty}{(ab, tx, ty; q)_\infty}
{_3\phi_2 \left({{d, xt, yt}\atop{qaxt, qbyt}}; q, qc \right)}.
\end{eqnarray*}
Thus, by Theorem~~\ref{pliuthm},  there exists a sequence $\{\lambda_n\}$ independent of $x$ and $y$ such that
 \[
\frac{(c, axt, byt; q)_\infty}{(ab, tx, ty; q)_\infty}
{_3\phi_2 \left({{d, xt, yt}\atop{axt, byt}}; q, c \right)}=\sum_{n=0}^\infty \lambda_n \Phi_n^{(a, b)}(x, y|q).
 \]
 Setting $y=0$ in this equation, using $\Phi_n^{(a, b)}( x, 0|q)=(a; q)_n x^n$, we find that
  \[
\frac{(c, axt; q)_\infty}{(ab, tx; q)_\infty}
{_2\phi_1 \left({{d, xt}\atop{axt}}; q, c \right)}=\sum_{n=0}^\infty \lambda_n (a; q)_n x^n.
 \]
 Using the Heine transformation formula and noting that $ab=cd$, we have
 \[
 \frac{(c, axt; q)_\infty}{(ab, tx; q)_\infty}
{_2\phi_1 \left({{d, xt}\atop{axt}}; q, c \right)}={_2\phi_1 \left({{a, c}\atop{ab}}; q, xt \right)}.
 \]
 Comparing the above two equations, we find that
 $
 (q, ab; q)_n \lambda_n=(c; q)_n t^n.
 $
 This completes the proof of Theorem~\ref{gthmc}.
 \qed
\end{proof}
%%%%%%%%%%%%%%%%%%%%%%%%%%%%%%%%%%%%%%%%%%%%%%%%%%%%%%%%%%%%%%%%%%%%%%%%%%%%%%%%
%%%%%%%%%%%%%%%%%%%%%%%%%%%%%%%%%%%%%%%%%%%%%%%%%%%%%%%%%%%%%%%%%%%%%%%%%%%%%%%%%%%%%%%%%%%%%%%%%
\section{The proof of Theorem~\ref{AndrewsLauricella}}
\label{sec:4}
%%%%%%%%%%%%%%%%%%%%%%%%%%%%%%%%%%%%%%%%%%%%%%%%%%%%%%%%%%%%%%%%%%%%%%%%%%%%%%%%%%%%%%%%%%%%%%%%%
\begin{proof}
\smartqed
If we use  $f(x_1, y_1, \ldots, x_k, y_k)$ to denote the right-hand side of the equation in Theorem~\ref{AndrewsLauricella}, then,
using the ratio test, we find that $f$ is an analytic function of
$x_1, y_1, \ldots, x_k, y_k$ for $\max\{|a|, |c|, |x_1|, |y_1|, \ldots, |x_k|, |y_k|\}<1.$
By a direct computation, we deduce that for $j=1, 2\ldots, k$,
\begin{eqnarray*}
&&\partial_{q, x_j}(1-\beta_j \eta_{y_j})\{f\}=\partial_{q, y_j}(1-\alpha_j \eta_{x_j})\{f\}\\
&&=\frac{(1-\alpha_j)(1-\beta_j)(a, \alpha_1 x_1, \beta_1 y_1, \ldots, q\alpha_j x_j, q\beta_j y_j, \ldots \alpha_k x_k, \beta_k y_k; q)_\infty}
{(c, x_1, y_1, \ldots, x_j, y_j, \ldots, x_k, y_k; q)_\infty}\\
&&\qquad \times{_{2k+1}\phi_{2k}}\left({{c/a, x_1, y_1, \ldots,  x_j, y_j, \ldots, x_k, y_k}
\atop{\alpha_1 x_1, \beta_1 y_1, \ldots, q\alpha_j x_j, q\beta_j y_j, \ldots, \alpha_k x_k, \beta_k y_k}}; q, qa\right).
\end{eqnarray*}
Thus, by  Theorem~\ref{pliuthm}, there exists a sequence $\{\lambda_{n_1, \ldots, n_k}\}$ independent
of $x_1, y_1, \ldots, x_k, y_k$ such that
\begin{eqnarray}
&&f(x_1, y_1, \ldots, x_k, y_k)\nonumber\\
&&=\sum_{n_1, \ldots, n_k=0}^\infty \lambda_{n_1, \ldots, n_k}
\Phi^{(\alpha_1, \beta_1)}_{n_1}(x_1, y_1|q) )\cdots \Phi^{(\alpha_k, \beta_k)}_{n_k}(x_k, y_k|q).
\label{mgeqn}
\end{eqnarray}
Setting $x_1=x_2\cdots=x_k=0$ in this equation, we immediately deduce that
\begin{eqnarray}
&&\sum_{n_1, \ldots, n_k=0}^\infty \lambda_{n_1, \ldots, n_k}
(\beta_1; q)_{n_1}\cdots (\beta_k; q)_{n_k}y_{1}^{n_1}\cdots y_{k}^{n_k}\nonumber\\
&&=\frac{(a, \beta_1y_1, \ldots, \beta_k y_k; q)_\infty}{(c, y_1, \ldots, y_k; q)_\infty}
{_{k+1}\phi_{k}}\left({{c/a, y_1, \ldots, y_k}
\atop{\beta_1 y_1, \ldots, \beta_k y_k}}; q, a\right).
\label{mgeqn1}
\end{eqnarray}
Applying the Andrews identity in Proposition~\ref{Andrewspp} to the right-hand side of the above equation, we obtain
\begin{eqnarray}
&&\sum_{n_1, \ldots, n_k=0}^\infty \lambda_{n_1, \ldots, n_k}
(\beta_1; q)_{n_1}\cdots (\beta_k; q)_{n_k}y_{1}^{n_1}\cdots y_{k}^{n_k}\nonumber\\
&&=\sum_{n_1, \ldots, n_k=0}^\infty \frac{(a; q)_{n_1+\cdots+n_k}(\beta_1; q)_{n_1}\cdots (\beta_k; q)_{n_k}
y_1^{n_1} \cdots y_k^{n_k}}
{(c; q)_{n_1+\cdots+n_k}(q; q)_{n_1}\cdots (q; q)_{n_k}}.
\label{mgeqn11}
\end{eqnarray}
Equating the coefficients of $y_{1}^{n_1}\cdots y_{k}^{n_k}$ on both sides of the equation, we find that
\[
\lambda_{n_1, \ldots, n_k}=\frac{(a; q)_{n_1+n_2+\cdots+n_k}}
{(c; q)_{n_1+\cdots+n_k}(q; q)_{n_1}(q; q)_{n_2}\cdots (q; q)_{n_k}}.
\]
Substituting this into (\ref{mgeqn}), we complete the proof of
Theorem~\ref{AndrewsLauricella}.
\qed
\end{proof}
 %%%%%%%%%%%%%%%%%%%%%%%%%%%%%%%%%%%%%%%%%%%%%%%%%%%%%%%%%%%%%%%%%%%%%%%%%%%%%%%%%%%%%%%%%%%%%%%%%
\section{Another multilinear generating function}
\label{sec:5}
%%%%%%%%%%%%%%%%%%%%%%%%%%%%%%%%%%%%%%%%%%%%%%%%%%%%%%%%%%%%%%%%%%%%%%%%%%%%%%%%%%%%%%%%%%%%%%%%%
In this section we will discuss some applications of Theorem~\ref{pliuthm} to $q$-beta integrals.
The Jackson $q$-integral of the function $f(x)$ from $a$ to $b$ is defined as
\[
\int_{a}^b f(x) d_q x=(1-q)\sum_{n=0}^\infty [b f(bq^n)-af(aq^n)]q^n.
\]
Using the $q$-integral notation, one can write some $q$-formulas in more compact forms.
For example, Sears' nonterminating extension of the $q$-Saalsch\"{u}tz summation can be
rewritten in the following beautiful form, which was first noticed by Al-Salam and Verma \cite{SalamVerma}.
\begin{prop}\label{salverpp} If there are no zero factors in the denominator of the integral
and $\max\{|a|, |b|, |cx|, |cy|, |ax/y|, |by/x|\}<1$, then, we have
 \[
 \int_{x}^y \frac{(qz/x, qz/y, abcz; q)_\infty}{(az/y, bz/x, cz; q)_\infty}d_q z
 =\frac{(1-q)y(q, x/y, qy/x,  ab, acx, bcy; q)_\infty}{(ax/y, by/x, a, b, cx, cy; q)_\infty}.
 \]
\end{prop}
Using this proposition we can find the following $q$-integral representation  for $\Phi_k^{(a, b)}(x, y|q)$.
\begin{prop} \label{qintrep} If there are no zero factors in the denominator of the integral, then, we have
\begin{eqnarray*}
&&\Phi_k^{(a, b)} (x, y|q)\\
&&=\frac{(ab; q)_k(a, b, by/x, ax/y; q)_\infty}{(1-q)y(q, ab, x/y, qy/x; q)_\infty}
\int_{x}^y \frac{(qz/x, qz/y; q)_\infty z^k}{(bz/x, az/y; q)_\infty}d_q z.
\end{eqnarray*}
\end{prop}
\begin{proof} \smartqed
Using the generating function for $\Phi_n^{(a, b)}(x, y|q)$ in Proposition~\ref{gthmb}, we
find that
\[
\frac{( acx, bcy; q)_\infty}{(cx, cy; q)_\infty}
=\sum_{n=0}^\infty \frac{\Phi_n^{(a, b)}(x, y|q)c^n}{(q; q)_n}.
\]
It follows that
\begin{eqnarray*}
&&\int_{x}^y \frac{(qz/x, qz/y, abcz; q)_\infty}{(az/y, bz/x, cz; q)_\infty}d_q z\\
&&=\frac{(1-q)y(q, x/y, qy/x,  ab; q)_\infty}{(ax/y, by/x, a, b; q)_\infty}
\sum_{n=0}^\infty \frac{\Phi_n^{(a, b)}(x, y|q)c^n}{(q; q)_n}.
\end{eqnarray*}
Applying the $q$-partial derivative operator $\mathcal{\partial}_{q, c}^n$ to act both sides of the  equation, we deduce that
\begin{eqnarray*}
&& (ab; q)_k\int_{x}^y \frac{(qz/x, qz/y, q^kabcz; q)_\infty}{(az/y, bz/x, cz; q)_\infty}d_q z\\
&&=\frac{(1-q)y(q, x/y, qy/x,  ab; q)_\infty}{(ax/y, by/x, a, b; q)_\infty}
\sum_{n=k}^\infty \frac{\Phi_n^{(a, b)}(x, y|q)c^{n-k}}{(q; q)_{n-k}}.
\end{eqnarray*}
Setting $c=0$ in this equation, we complete the proof of Proposition~\ref{qintrep}.  \qed
\end{proof}
Putting $a=b=0$ in Proposition~\ref{qintrep}, we immediately find the following $q$-integral representation for
the the homogeneous Rogers--Szeg\H{o} polynomials.
\begin{prop}\label{rsqrpp} We have
\[
h_k(x, y|q)=\frac{1}{(1-q)y(q, x/y, qy/x; q)_\infty}
\int_{x}^y (qz/x, qz/y; q)_\infty z^k d_q z.
\]
\end{prop}

Using Proposition~\ref{qintrep} and  Theorem~\ref{pliuthm}, we can prove the following
theorem.
\begin{thm}\label{mulgenfun}
For $\max_{j\in\{1,\ldots,k\}}\{ |xt|, |yt|, |xu_j|, |xv_j|, |yu_j|, |yv_j|\}<1,$ we have
\begin{eqnarray*}
&&\sum_{n_1, \ldots, n_k=0}^\infty
\frac{\Phi_{n_1}^{(\alpha_1, \beta_1)}(u_1, v_1|q)\cdots \Phi_{n_k}^{(\alpha_k, \beta_k)}(u_k, v_k|q)}{(q; q)_{n_1} \cdots (q; q)_{n_k}}\\
&&\qquad \times \int_{x}^y  z^{n_1+\cdots+n_k} \frac{(qz/x, qz/y, \gamma z t; q)_\infty}{(bz/x, az/y, zt; q)_\infty}  d_q z\\
&&=\int_{x}^y \frac{(qz/x, qz/y, \gamma zt, \alpha_1 u_1 z, \beta_1 v_1 z, \ldots, \alpha_k u_k z, \beta_k v_k z; q)_\infty}
{(bz/x, az/y, zt, zu_1, zv_1, \ldots, zu_k, zv_k; q)_\infty}d_q z.
\end{eqnarray*}
\end{thm}
\begin{proof} \smartqed
If we use  $f(u_1, v_1, \ldots, u_k, v_k)$ to denote the right-hand side of the equation in Theorem~\ref{mulgenfun}, then,
using the ratio test, we find that $f$ is analytic function at $(0, \ldots, 0)\in \mathbb{C}^{2k}$.
For simplicity, we  use $I(\alpha_1, \beta_1, \ldots, \alpha_k, \beta_k; z)$ to denote the integrand function.

By a direct computation, we easily deduce that for $j=1, 2\ldots, k$,
\begin{eqnarray*}
&&\partial_{q, u_j}(1-\beta_j \eta_{v_j})\{f\}=\partial_{q, v_j}(1-\alpha_j \eta_{u_j})\{f\}\\
&&=(1-\alpha_j)(1-\beta_j)
\int_{x}^y z{I(\alpha_1, \beta_1, \ldots, q\alpha_j, q\beta_j, \ldots \alpha_k, \beta_k; z)}d_q z.
\end{eqnarray*}
Thus, by  Theorem~\ref{pliuthm},  there exists a sequence $\{c_{n_1, \ldots, n_k}\}$ independent
of $u_1, v_1, \ldots, u_k, v_k$ such that
\begin{eqnarray}
&&f(u_1, v_1, \ldots, u_k, v_k) \nonumber\\
&&=\sum_{n_1, \ldots, n_k=0}^\infty c_{n_1, \ldots, n_k}
\Phi^{(\alpha_1, \beta_1)}_{n_1}(u_1, v_1|q) )\cdots \Phi^{(\alpha_k, \beta_k)}_{n_k}(u_k, v_k|q).
\label{mulgen1}
\end{eqnarray}
Putting $v_1=\cdots=v_k=0$ in this equation and noting that
$\Phi^{(\alpha_j, \beta_j)}_{n_j}(u_j, 0|q)=(\alpha_j; q)_{n_j}u_j^{n_j}$ for $j=1, 2\ldots, k$,
we find that
\begin{eqnarray*}
&&\sum_{n_1, \ldots, n_k=0}^\infty c_{n_1, \ldots, n_k} (\alpha_1; q)_{n_1}\cdots (\alpha_k; q)_{n_k}u_1^{n_1}
\cdots u_{k}^{n_k}\\
&&=\int_{x}^y \frac{(qz/x, qz/y, \gamma zt, \alpha_1 u_1 z, \ldots, \alpha_k u_k z; q)_\infty}
{(bz/x, az/y, zt, zu_1, \ldots, zu_k; q)_\infty}d_q z.
\end{eqnarray*}
Applying the operator $\mathcal{\partial}_{q, u_1}^{n_1}\cdots \mathcal{\partial}_{q, u_k}^{n_k}$ to act both sides of the  equation
and then setting $u_1=\cdots=u_k=0$, we deduce that
\begin{eqnarray*}
&&c_{n_1, \ldots, n_k} (\alpha_1; q)_{n_1}\cdots (\alpha_k; q)_{n_k} (q; q)_{n_1}\cdots (q; q)_{n_k}\\
&&=(\alpha_1; q)_{n_1}\cdots (\alpha_k; q)_{n_k}
\int_{x}^y z^{n_1+\cdots+n_k} \frac{(qz/x, qz/y, \gamma zt; q)_\infty}
{(bz/x, az/y, zt; q)_\infty}d_q z.
\end{eqnarray*}
It follows that
\begin{eqnarray*}
&&c_{n_1, \ldots, n_k} (q; q)_{n_1}\cdots (q; q)_{n_k}\\
&&=
\int_{x}^y z^{n_1+\cdots+n_k} \frac{(qz/x, qz/y, \gamma zt; q)_\infty}
{(bz/x, az/y, zt; q)_\infty}d_q z.
\end{eqnarray*}
Substituting this equation into (\ref{mulgen1}), we complete the proof of Theorem~\ref{mulgenfun}.
\qed
\end{proof}
Setting $a=b=0$ in Theorem~\ref{mulgenfun} and then equating the coefficients of $t^m$ on both sides of the resulting
equation, we conclude the following proposition, which is equivalent to \cite[Theorem~7.3]{LiuHahn}.
\begin{prop}\label{amulgenfun}
For $\max_{j\in\{1,\ldots,k\}}\{ |xt|, |yt|, |xu_j|, |xv_j|, |yu_j|, |yv_j|\}<1,$ we have
\begin{eqnarray*}
&&\sum_{n_1, \ldots, n_k=0}^\infty
\frac{\Phi_{n_1}^{(\alpha_1, \beta_1)}(u_1, v_1|q)\cdots \Phi_{n_k}^{(\alpha_k, \beta_k)}(u_k, v_k|q)}{(q; q)_{n_1} \cdots (q; q)_{n_k}}\\
&&\qquad \times \int_{x}^y  z^{m+n_1+\cdots+n_k}  (qz/x, qz/y; q)_\infty   d_q z\\
&&=\int_{x}^y \frac{z^m (qz/x, qz/y, \alpha_1 u_1 z, \beta_1 v_1 z, \ldots, \alpha_k u_k z, \beta_k v_k z; q)_\infty}
{( zu_1, zv_1, \ldots, zu_k, zv_k; q)_\infty}d_q z.
\end{eqnarray*}
\end{prop}
%%%%%%%%%%%%%%%%%%%%%%%%%%%%%%%%%%%%%%%%%%%%%%%%%%%%%%%%%%%%%%%%%%%%%%%%%%%%%%%%%%%%%%%%%%%%%%%%%%%%%%%%%%%%%%%%%
\section{A $q$-integral formula}
\label{sec:6}
%%%%%%%%%%%%%%%%%%%%%%%%%%%%%%%%%%%%%%%%%%%%%%%%%%%%%%%%%%%%%%%%%%%%%%%%%%%%%%%%%%%%%%%%%%%%%%%%%%%%%%%%%%%%%%%%%%

The Andrews--Askey integral formula \cite[Theorem~1]{Andrews+Askey} is stated in the following proposition.
\begin{prop}\label{andaskpp} If there are no zero factors in the denominator of the integral, then,
we have
\begin{equation}
\int_{u}^v \frac{(qx/u, qx/v; q)_\infty}{(cx, dx; q)_\infty}d_q x
=\frac{(1-q)v(q, u/v, qv/u, cduv; q)_\infty}{(cu, cv, du, dv; q)_\infty}.
\label{qinteqn1}
\end{equation}
\end{prop}
Applying ${\partial}_{q, c}^n$ to act both sides of (\ref{qinteqn1}) and then using the $q$-Leibniz rule, one can easily find the following proposition \cite{Wang2009}.
\begin{prop}\label{wangpp}
 If there are no zero factors in the denominator of the integral, then,
we have
\begin{eqnarray*}
\int_{u}^v \frac{x^n(qx/u, qx/v; q)_\infty}{(cx, dx; q)_\infty}d_q x
&&=\frac{(1-q)v(q, u/v, qv/u, cduv; q)_\infty}{(cu, cv, du, dv; q)_\infty}\\
&&\qquad \times \sum_{j=0}^n {n\brack j}_q \frac{(cv, dv; q)_j}{(cduv; q)_j}u^jv^{n-j}.
\end{eqnarray*}
\end{prop}
The main result of this section is the following $q$-integral formula.
\begin{thm}\label{liuqint}If there are no zero factors in the denominator of the integral, then,
we have
\end{thm}
\begin{eqnarray*}
&&\int_{u}^v \frac{(qx/u, qx/v, \alpha ax, \beta b x; q)_\infty}{(ax, bx, cx, dx; q)_\infty} d_q x\\
&&=\frac{(1-q)v(q, u/v, qv/u, cduv; q)_\infty }{(cu, cv, du, dv; q)_\infty } \\
&&\qquad \times \sum_{n=0}^\infty \frac{ \Phi_n^{(\alpha, \beta)}(a, b|q)}{(q; q)_n}
\sum_{j=0}^n {n\brack j}_q \frac{(cv, dv; q)_j}{(cduv; q)_j}u^jv^{n-j}.
\end{eqnarray*}
\begin{proof} \smartqed
If we use $I(a, b)$ to  denote the $q$-integral in the above theorem, then it is
easy to show that $I(a, b)$ is analytic near $(0, 0)\in \mathbb {C}^2.$ A straightforward
evaluation shows that
\begin{eqnarray*}
&&\partial_{q, a} \{I(a, b)-\beta I(a, bq)\}=\partial_{q, b}\{I(a, b)-\alpha I(aq, b)\}\\
&&=(1-\alpha)(1-\beta)\int_{u}^v \frac{x(qx/u, qx/v, \alpha qax, \beta q bx; q)_\infty}{(ax, bx, cx, dx; q)_\infty} d_q x.
\end{eqnarray*}
Thus, by Theorem~\ref{pliuthm}, there exists a sequence $\{\lambda_n\}$ independent of $a$ and $b$ such that
\begin{equation}
I(a, b)=\int_{u}^v \frac{(qx/u, qx/v, \alpha ax, \beta b x; q)_\infty}{(ax, bx, cx, dx; q)_\infty} d_q x=
\sum_{n=0}^\infty \lambda_n \Phi_n^{(\alpha, \beta)} (a, b|q).
\label{qinteqn2}
\end{equation}
Setting $b=0$ in the above equation and using $\Phi_n^{(\alpha, \beta)} (a, 0|q)=(\alpha; q)_n a^n$, we immediately deduce that
\[
\int_{u}^v \frac{(qx/u, qx/v, \alpha a x; q)_\infty}{( ax, cx, dx; q)_\infty} d_q x=
\sum_{n=0}^\infty \lambda_n   (\alpha; q)_n a^n.
\]
Applying ${\partial}_{q, a}^n$ to act both sides of the  equation, setting $a=0$,
and using Proposition~\ref{wangpp}, we obtain
\begin{eqnarray*}
\lambda_n&&=\frac{(1-q)v(q, u/v, qv/u, cduv; q)_\infty}
{(cu, cv, du, dv; q)_\infty (q; q)_n}\\
&&\qquad \times \sum_{j=0}^n {n\brack j}_q \frac{(cv, dv; q)_j}{(cduv; q)_j}u^jv^{n-j}.
\end{eqnarray*}
Substituting the above equation into (\ref{qinteqn2}), we complete the proof Theorem~\ref{liuqint}.
\qed
\end{proof}

\section*{Acknowledgments}
This paper is dedicated to Professor Krishnaswami Alladi on the occasion of his $60$th birthday.  I am grateful to the anonymous referee for careful reading of the manuscript and many invaluable suggestions
and comments. This work was supported by the National Natural Science Foundation of China (Grant No. 11571114).

\end{document}